\documentclass[11pt, leqno]{scrartcl}

\usepackage[english]{babel}
\usepackage{csquotes}

\usepackage{amssymb}
\usepackage{amsthm}
\usepackage{amsmath}
\usepackage{mathrsfs}

\usepackage[a4paper, margin=2.8cm]{geometry}
\usepackage{color}
\usepackage{hyperref}
\usepackage{numprint}

\usepackage{enumerate}

\theoremstyle{plain}
\newtheorem{theorem}{Theorem}

\newtheorem{lemma}[theorem]{Lemma}
\newtheorem{proposition}[theorem]{Proposition}
\newtheorem{corollary}[theorem]{Corollary}

\theoremstyle{definition}
\newtheorem{definition}[theorem]{Definition}

\theoremstyle{remark}
\newtheorem{remark}[theorem]{Remark}

\numberwithin{equation}{section}
\numberwithin{theorem}{section}

\newcommand{\longto}{\longrightarrow}

\newcommand{\CC}{\mathbb{C}}
\newcommand{\FF}{\mathbb{F}}
\newcommand{\HH}{\mathbb{H}}
\newcommand{\QQ}{\mathbb{Q}}
\newcommand{\ZZ}{\mathbb{Z}}

\newcommand{\cF}{\mathcal{F}}
\newcommand{\cH}{\mathcal{H}}
\newcommand{\diag}{\operatorname{diag}}
\newcommand{\fp}{\FF_p}

\newcommand{\GL}{\operatorname{GL}}
\newcommand{\Gr}{\operatorname{Gr}}
\newcommand{\GSp}{\operatorname{GSp}}
\newcommand{\im}{\operatorname{Im}}
\newcommand{\Mat}{\operatorname{Mat}}
\newcommand{\one}{\mathbf{1}}
\newcommand{\qbinom}{\binom}

\newcommand{\Sym}{\operatorname{Sym}}
\newcommand{\Sp}{\operatorname{Sp}}
\newcommand{\T}{\mathcal{T}}
\newcommand{\Tr}{\operatorname{Tr}}
\newcommand{\tT}{\widetilde{\T}}

\title{Hecke eigenvalues of Siegel modular forms of ``different weights''}
\author{
  Alexandru Ghitza and Robert Sayer
}

\begin{document}

\maketitle

\begin{abstract}
  Given two Siegel eigenforms of different weights, we determine explicit
  sets of Hecke eigenvalues for the two forms that must be distinct.  In
  degree two, and under some additional conditions, we determine explicit sets
  of Fourier coefficients of the two forms that must be distinct.

  \emph{Keywords:} Siegel modular forms; Hecke operators; Fourier
  coefficients.
\end{abstract}

\section{Introduction}

Our story revolves around the

\begin{center}
  {\bf Motivating question:} How similar can two Hecke eigenforms of
    different weights be?
\end{center}

The case of elliptic modular forms has been investigated in some depth, both
in characteristic zero and in positive characteristic; see~\cite{Murty},
\cite{Kohnen}, \cite{Ghitza}, \cite{ChowGhitza}.  Here the eigenforms are
compared via the Fourier coefficients, or via the Hecke eigenvalues, which are
equivalent (after normalization).

In this paper we treat the case of Siegel modular forms in characteristic
zero.  While this is motivated by the elliptic case, the richness of the
Siegel theory makes for a significantly more diverse picture.  On one hand,
Fourier coefficients encode more information than the Hecke eigenvalues: while
the Fourier expansion completely determines the form for a fixed weight and
level (by a generalization of the $q$-expansion
principle~\cite[Proposition~V.1.8]{FaltingsChai}), the set of Hecke
eigenvalues does not determine the form--there exist distinct eigenforms of
the same level and weight that have the same Hecke eigenvalues for $p$ not
dividing the level (see the introduction of~\cite{Schulze-Pillot}).  On the
other hand, even when staying entirely on the local Hecke algebra side, the
structure is delicate enough that there are several choices of Hecke operators
to consider: the ``standard'' generators $T(p)$, $T_1(p^2),\ldots,T_n(p^2)$
described by Andrianov~\cite{AndrianovZhuravlev}; the ``averaged'' generators
$\mathcal{T}(p)$, $\tT_1(p^2),\ldots,\tT_n(p^2)$ defined and studied by Hafner
and Walling~\cite{HafnerWalling}; or the operators $T(p^r)$ that appear more
naturally in geometric situations such as the work of Bergstr\"om, Faber and
van der Geer~\cite{BergstromFaberGeer-deg3}.

Our investigation touches upon most of these aspects, with varying degrees of
success and generality.  The most direct way of approaching our
{\bf Motivating question} is via the operator $T_n(p^2)$: this acts on a
Siegel modular form of level coprime to $p$ and weight given by
$(\lambda_1\geq\ldots\geq\lambda_n)$ as multiplication by the scalar
$p^{\sum\lambda_j-n(n+1)}$.  It follows that two forms that have the same
eigenvalue for $T_n(p^2)$ must have the same $\sum\lambda_j$.  
The contrapositive version of this statement is:
if for two forms $F$ and $G$ of level coprime to $p$ the corresponding
integers $\sum\lambda_j$ differ (this is the sense in which the ``different
weights'' of the title should be understood), then their eigenvalues under the
operator $T_n(p^2)$ must also differ.  This basic observation forms the basic
leitmotif of the paper, and the variations thereupon may be described
informally as follows:
\begin{itemize}
  \item in Section~\ref{sect:tjp2}, we show that the eigenvalues of $F$ and
    $G$ for at least one of the operators $T(p)$, $T(p^2)$,
    $T_1(p^2),\ldots,T_{n-1}(p^2)$ must be distinct;
  \item in Section~\ref{sect:tpr}, we consider the special case of degree $2$
    and show that the eigenvalues of $F$ and $G$ for at least one of the
    operators $T(p^r)$, $r=1,\ldots,6$, must be distinct;
  \item in Section~\ref{sect:fourier}, we show that, subject to a number of
    conditions, there exists a Fourier expansion index $S$ of
    explicitly-bounded determinant such that the coefficient of $F$
    at $S$ is distinct from the coefficient of $G$ at $S$.
\end{itemize}

These results are preceded by a review of the basic theory of Hecke operators
on Siegel modular forms in Section~\ref{sect:review}, and by the derivation of
a formula for $T(p^2)$ (inspired by work of Hafner-Walling) in
Section~\ref{sect:tp2}.

{\bf Acknowledgements:} We thank G. van der Geer, M. Raum and R.
Schulze-Pillot for answering our various questions.  The first author was
supported by Discovery Grant DP120101942 from the Australian Research Council.

\section{Hecke operators on Siegel modular forms}
\label{sect:review}

We gather here some definitions and basic results on Siegel modular forms and
their Hecke action.  For more leisurely expositions of various parts of this
material, the reader is invited to consult~\cite{Geer} or~\cite{Freitag}.

\subsection{Siegel modular forms}
Let $n\geq 1$ be an integer and let $R$ be a commutative ring. 
Consider the $R$-module $R^{2n}$
with generators $e_1,\ldots,e_n$, $f_1,\ldots,f_n$, equipped with a
\emph{symplectic form} $\langle\cdot, \cdot \rangle$ which is defined on the
generators by the rules 
\begin{equation*}
  \langle e_i,e_j \rangle = 0, \quad \langle f_i,f_j
  \rangle = 0 \quad \mbox{and} \quad \langle e_i,f_j \rangle =
  \delta_{ij}~\mbox{(Kronecker's delta)} 
\end{equation*}
and extended by $R$-bilinearity.

\begin{definition}
The \emph{symplectic group} $\Sp(2n,R)$ is the
automorphism group of the pair $(R^{2n}, \langle\cdot, \cdot\rangle)$:
\begin{equation*} 
  \Sp(2n,R)=\left\{\alpha \in \GL(R^{2n})\mid
    \langle\alpha(u), \alpha(v)\rangle=\langle u, v\rangle
    \text{ for all }u,v\in R^{2n}\right\}.
\end{equation*}
We also work with the \emph{group of symplectic similitudes}
\begin{align*} 
  \GSp(2n,R)=\{\alpha \in \GL(R^{2n})\mid 
    & \text{ there exists }
    \eta(\alpha)\in R^\times\text{ such that }\\
    &\langle\alpha(u), \alpha(v)\rangle=\eta(\alpha)\langle u, v\rangle
    \text{ for all }u,v\in R^{2n}\}.
\end{align*}
\end{definition}

With respect to the basis comprising the $e_i$ and $f_i$, the elements of
$\Sp(2n,R)$ and $\GSp(2n,R)$ are represented by matrices
\begin{align*}
  \Sp(2n, R) &= \left\{\gamma\in\GL(2n, R)\mid \gamma J\gamma^t = J\right\}\\
  \GSp(2n, R) &= \left\{\gamma\in\GL(2n, R)\mid \gamma J\gamma^t =
    \eta(\gamma)J\text{ for some }\eta(\gamma)\in R^\times\right\},\\
  \intertext{where}
  J&=\begin{pmatrix}
    0 & I_n\\ -I_n & 0
  \end{pmatrix}.
\end{align*}

The \emph{Siegel modular group} of degree $n$ is $\Gamma=\Sp(2n,\ZZ)$.
We shall be principally concerned with a family of so-called \emph{congruence
  subgroups} of $\Gamma$.

\begin{definition}
Let $N\geq 1$ be an integer.  The \emph{congruence subgroup} $\Gamma_0(N)$ of
level $N$ is
\begin{equation*}
\Gamma_0(N)=\left\{
  \gamma=\begin{pmatrix}
 A & B \\ C & D
\end{pmatrix}\in \Gamma~\Big|~C\equiv 0 \pmod{N}\right\}.
\end{equation*}
\end{definition}

\begin{definition} 
Let $n\geq 1$ be an integer.  The \emph{Siegel upper half space}
$\HH^n$ is the set of symmetric $n\times n$ complex matrices having
positive-definite imaginary part:
\begin{equation*}
  \HH^n=\left\{z\in
  \Mat(n,\CC)\mid z^t=z\text{ and }
  \im(z)>0\right\}
\end{equation*}
The set $\HH^n$ is an open subset of the complex manifold
$\Mat(n,\CC)\cong \CC^{n^2}$.
\end{definition}

The group $\Gamma$ acts on $\HH^n$ by generalized M\"obius transformations:
\begin{equation*}
  \begin{pmatrix}
 A & B \\ C & D
\end{pmatrix}\cdot z := (Az + B)(Cz + D)^{-1}.
\end{equation*}

\begin{definition} 
  Let $n\geq 2$ and $N\geq 1$ be integers and let $\rho\colon\GL_n(\CC)\to \GL(V)$
  be a polynomial\footnote{A finite-dimensional representation
    $\rho\colon\GL_n(\CC)\to\GL(V)$ is \emph{polynomial} if it is given by
  polynomial functions in the entries of the matrices in $\GL_n(\CC)$.} 
  representation. A \emph{Siegel modular form of
    degree $n$, weight $\rho$ and level $N$} is a holomorphic function
  $F\colon\HH^{n}\to V$ such that
  \begin{equation*}
  F(\gamma\cdot z)=\rho(Cz+D)F(z)
\end{equation*}
for all $\gamma=\begin{pmatrix}
 A & B \\ C & D 
\end{pmatrix}\in \Gamma_0(N)$ and all $z\in\HH^n$.  The $\CC$-vector space of
all such forms is denoted $M_\rho(\Gamma_0(N))$.
\end{definition}

\begin{remark}
  The case $n=1$ requires an extra condition, holomorphicity at the cusps,
  which is automatically satisfied for $n\geq 2$ (by the K\"ocher principle).
\end{remark}

\begin{remark}
  We may commit various abuses of notation regarding the weight $\rho$ of a
  Siegel modular form:
  \begin{itemize}
    \item Finite-dimensional representations $\rho\colon\GL_n(\CC)\to\GL(V)$ are indexed
by nonincreasing $n$-tuples of integers
$(\lambda_1\geq\lambda_2\geq\ldots\geq\lambda_n)$, and we will refer to such a
tuple as the weight.
\item The special case
$\lambda_1=\lambda_2=\ldots=\lambda_n=k$ corresponds to \emph{scalar-valued}
Siegel modular forms; here we call $k$ the weight.
\end{itemize}
\end{remark}

\subsection{The Hecke algebra}
\label{sect:hecke_algebra}

We recall the construction of some interesting elements of the local Hecke
algebra of $\Sp(2n)$ at a prime $p$; for this, we combine the approach and
notation from~\cite[Section 2]{Gross} and~\cite[Section 16]{Geer}.

Let $G=\Sp(2n, \QQ_p)$ and $K=\Sp(2n, \ZZ_p)$.  The \emph{(local) Hecke
  algebra} $\cH$ is the Hecke algebra of the pair $(G, K)$, i.e.
\begin{equation*}
  \cH = \left\{f\colon G\longto\ZZ\mid f\text{ locally constant, compactly
    supported, bi-$K$-invariant}\right\}.
\end{equation*}
The multiplication is given by convolution of such functions; for all $x\in
G$, we set
\begin{equation*}
  (f\cdot g)(x)=\int_{\gamma\in G} f(\gamma)g(\gamma^{-1}x)\,d\gamma,
\end{equation*}
where $d\gamma$ is the unique Haar measure on $G$ which is normalized so that $K$
has volume $1$.

The prototypical examples of elements of $\cH$ are provided by the
characteristic functions of double cosets $K\gamma K$ with $\gamma\in G$:
\begin{equation*}
  \one_{K\gamma K}(x) = \begin{cases}
    1 & \text{if }x\in K\gamma K\\
    0 & \text{otherwise}.
  \end{cases}
\end{equation*}
In fact, every element of $\cH$ is a finite $\ZZ$-linear combination of
characteristic functions $\one_{K\gamma K}$.

For any $r\geq 1$, consider the set of matrices
\begin{equation*}
  O_n(p^r)=\{\gamma\in\Mat(2n,\CC)\mid \gamma^tJ\gamma=p^rJ\}.
\end{equation*}

\begin{definition}
  \label{def:hecke_ops}
  Set
  \begin{align}
    \label{eq:tp}
    T(p) &=\one_{K\gamma K}\qquad\text{where }
    \gamma=\diag(\overbrace{1,\ldots,1}^n,
    \overbrace{\vphantom{1}p,\ldots,p}^n)\\
    \intertext{and, for $j=0,\ldots,n$:}
    \label{eq:tjp2}
    T_j(p^2) &=\one_{K\gamma K}\qquad\text{where }
    \gamma=\diag(\overbrace{\vphantom{p^2}1,\ldots,1}^{n-j},
    \overbrace{\vphantom{p^2}p,\ldots,p}^j,
    \overbrace{p^2,\ldots,p^2}^{n-j},
    \overbrace{\vphantom{p^2}p,\ldots,p}^j),\\
    \intertext{and finally, for $r\geq 1$:}
    \label{eq:tm}
    T(p^r) &= \sum_{\gamma\in O_n(p^r)} \one_{K\gamma K}
  \end{align}
\end{definition}

(Equation~\eqref{eq:tm} is consistent with Equation~\eqref{eq:tp} when $r=1$.)

The algebra $\cH$ can be written as $\cH=\cH^0\left[1/T_n(p^2)\right]$, where
$\cH^0$ is the subalgebra generated by the characteristic functions of double
cosets of matrices with entries in $\ZZ_p$.  Moreover, a set of generators for
the algebra $\cH^0$ is given by the elements $T(p)$,
$T_1(p^2),T_2(p^2),\ldots,T_n(p^2)$.

\subsection{Action of the Hecke algebra on Siegel modular forms}
Let $F$ be a Siegel modular form of level $\Gamma_0(N)$ and weight $\rho$ given by
$(\lambda_1\geq\ldots\geq\lambda_n)$.  Given $\gamma=\begin{pmatrix}A & B\\C &
  D\end{pmatrix}\in\GSp(2n,\QQ)$ with positive determinant, we set
\begin{equation}
  \label{eq:slash}
  F|_{\gamma,\rho}(z)=
  \eta(\gamma)^{\sum\lambda_j-n(n+1)/2}\rho(Cz+D)^{-1}
  F(\gamma\cdot z).
\end{equation}

Let $p$ be a prime not dividing $N$ and let $K=\Sp(2n,\ZZ_p)$ as in
Section~\ref{sect:hecke_algebra}.
Given a double coset $K\gamma K$ and its decomposition into right cosets
\begin{equation*}
  K\gamma K=\coprod_{i=1}^h K\gamma_i,
\end{equation*}
we set
\begin{equation}
  \label{eq:double_coset_action}
  (K\gamma K)(F)(z) = \sum_{i=1}^h F|_{\gamma_i,\rho}(z).
\end{equation}
The result is a Siegel modular form of the same weight and level as $F$, and 
independent of the choice of double coset representative
$\gamma$ and right coset representatives $\gamma_i$.

Finally, we extend~\eqref{eq:double_coset_action} by $\ZZ$-linearity to an
action of the local Hecke algebra $\mathcal{H}$ on $M_\rho(\Gamma_0(N))$.  In
particular, we can think of the elements $T(p)$, $T_j(p^2)$, $T(p^r)$ from
Definition~\ref{def:hecke_ops} as operators on the space
$M_\rho(\Gamma_0(N))$.

In order to obtain explicit expressions for the action of these Hecke
operators, we need explicit decompositions of the relevant double cosets into
right cosets.  Such decompositions for the generators $T(p)$,
$T_1(p^2),\ldots,T_n(p^2)$ are given\footnote{%
The reader is cautioned that the formula for $T(p)$ appearing
in~\cite[Lemma~III.3.32]{AndrianovZhuravlev} has a small but unfortunate typo: in
Equation~(3.59), $p^2D^*$ should be replaced by $pD^*$; the formula is correct
in~\cite[Lemma~3.3.2]{Andrianov}, the former incarnation of the book.}
in~\cite[Lemma~III.3.32]{AndrianovZhuravlev}.
We will only need the
decomposition of $T_n(p^2)$, which
takes a particularly simple form and leads to the following description:
\begin{lemma}
  \label{lem:tnp2}
  Let $F\in M_\rho(\Gamma_0(N))$, with $\rho$ given by
  $(\lambda_1\geq\ldots\geq\lambda_n)$.  Let $p$ be a prime not dividing $N$.
  Then 
  \begin{equation*}
    T_n(p^2)F = p^{\sum\lambda_j-n(n+1)/2} F.
  \end{equation*}
\end{lemma}
\begin{proof}
  We know that $T_n(p^2)$ is given by the double coset
  \begin{equation*}
    KpI_{2n}K=KpI_{2n}=K\gamma\qquad\text{with }\gamma=
    \begin{pmatrix}
      pI_n & 0\\0 & pI_n
    \end{pmatrix}.
  \end{equation*}
  The claim follows from Equation~\eqref{eq:slash}, since:
  \begin{equation*}
    \eta(pI_{2n})=p^2,\qquad
    \rho(pI_n)=p^{\sum\lambda_j},\qquad\text{and}\qquad
    (pI_n)\cdot z=z.
  \end{equation*}
\end{proof}

\section{A formula for $T(p^2)$}
\label{sect:tp2}

In~\cite{HafnerWalling}, Hafner and Walling introduced a new set of generators
$\T(p)$, $\tT_j(p^2)$ ($j=1,\ldots,n$) for the local Hecke algebra at $p$.
Their motivation was that these generators act on the Fourier expansions of
Siegel modular forms in a much simpler way than the standard Hecke operators.
We are interested in the $\tT_j(p^2)$ because they satisfy the simple
relation~\eqref{eq:tilde_relation} given below; we use this to deduce the
formula~\eqref{eq:main_relation} for $T(p^2)$, which will play a crucial role
in Section~\ref{sect:tjp2}.

\begin{definition}
  \label{def:tilde_t}
  For $n\geq 1$, $j=0,1,\ldots,n$ and $k\geq n+1$, set
  \begin{align*}
    \T(p) &= p^{n(k-n-1)/2}T(p)\\
    \T_j(p^2) &= T_{n-j}(p^2)\\
    \tT_j(p^2) &= p^{j(k-n-1)}\sum_{t=0}^j \qbinom{n-t}{j-t}_p\T_t(p^2)
  \end{align*}
  where
  \begin{equation*}
    \qbinom{m}{\ell}_p=\#\Gr(\ell,m)(\fp)=
    \prod_{i=1}^\ell\frac{p^{m-\ell+i}-1}{p^i-1}
  \end{equation*}
  is the number of $\ell$-dimensional subspaces of $\fp^m$.
\end{definition}

\begin{remark}
  Hafner and Walling~\cite{HafnerWalling} write the operators
  $\T_j(p^2)$, etc., simply as $T_j(p^2)$.  We have preferred to differentiate
  them typographically from the standard generators.  We have also chosen to
  replace the notation $\beta_p(m,\ell)$ from~\cite{HafnerWalling} with the
  more established $q$-binomial coefficient notation $\qbinom{m}{\ell}_p$,
  which has the additional advantage that it is rather suggestive of the properties
  of these numbers that we will exploit below.
\end{remark}

\begin{remark}
Note that
\begin{equation*}
  \tT_n(p^2)=p^{n(k-n-1)}\sum_{t=0}^j T_{n-t}(p^2)=p^{n(k-n-1)}T(p^2).
\end{equation*}
\end{remark}

The operators $\T(p)$ and $\tT_j(p^2)$ satisfy the following relation:
\begin{theorem}[Hafner-Walling {\cite[Proposition 5.1]{HafnerWalling}}]
  \label{thm:tilde_relation}
  \begin{equation}
    \label{eq:tilde_relation}
    \tT_n(p^2)=\T(p)^2-\sum_{j=0}^{n-1}
    p^{k(n-j)+j(j+1)/2-n(n+1)/2} \tT_j(p^2).
  \end{equation}
\end{theorem}

We will make crucial use of this identity, which we first translate into a
statement about the operators $T(p)$, $T(p^2)$, $T_j(p^2)$:
\begin{theorem}
  \label{thm:main_relation}
  The following relation holds in the local Hecke algebra at $p$:
  \begin{equation}
    \label{eq:main_relation}
    T(p^2)=T(p)^2-\sum_{s=1}^n c_s T_s(p^2),
  \end{equation}
  where the coefficients $c_s$ are positive and given by
  \begin{equation}
    \label{eq:main_coeffs}
    c_s = \sum_{i=1}^s p^{i(i+1)/2}\qbinom{s}{i}_p=
    \left(\prod_{i=1}^s \left(p^i+1\right)\right) - 1.
  \end{equation}
\end{theorem}
\begin{proof}
  Combine Equation~\eqref{eq:tilde_relation} with the relations from
  Definition~\ref{def:tilde_t} and divide by the normalizing factor
  $p^{n(k-n-1)}$ to get
  \begin{equation*}
    T(p^2)-T(p)^2=
    -\sum_{j=0}^{n-1}p^{(n-j)(n-j+1)/2}
    \sum_{t=0}^j \qbinom{n-t}{j-t}_pT_{n-t}(p^2).
  \end{equation*}
  We substitute $i=n-j$ and $s=n-t$:
  \begin{equation*}
    T(p^2)-T(p)^2=
    -\sum_{i=1}^n p^{i(i+1)/2}\sum_{s=i}^n \qbinom{s}{s-i}_pT_s(p^2),
  \end{equation*}
  after which we interchange the two summations and use
  $\qbinom{s}{s-i}_p=\qbinom{s}{i}_p$ to get Equation~\eqref{eq:main_relation}
  with coefficients
  \begin{equation*}
    c_s = \sum_{i=1}^s p^{i(i+1)/2}\qbinom{s}{i}_p.
  \end{equation*}

  It remains to show that these coefficients can be simplified to give the
  product on the right hand side of Equation~\eqref{eq:main_coeffs}.  For this
  we use Gauss's binomial formula (see, for instance,
  \cite[Equation~5.5]{KacCheung}):
  \begin{equation*}
    (x+a)_q^n=\sum_{j=0}^n\qbinom{n}{j}_q q^{j(j-1)/2}a^jx^{n-j}.
  \end{equation*}
  With our notation ($x=1$, $a=p$, $q=p$, $n=s$), this gives precisely
  \begin{equation*}
    c_s+1 = \sum_{i=0}^s p^{i(i+1)/2}\qbinom{s}{i}_p=
    (1+p)_p^s = \prod_{j=1}^s \left(1+p^j\right),
  \end{equation*}
  which concludes the proof.
\end{proof}


\section{Distinguishing eigenforms via the operators $T(p)$, $T(p^2)$, $T_j(p^2)$}
\label{sect:tjp2}

Given a Siegel eigenform $F$ of degree $n$, let $a_F(T)$ denote its eigenvalue
for $T\in\mathcal{H}$, and set
\begin{equation*}
  \mathcal{E}_F(p) = \left(a_F(T(p)), a_F(T(p^2)), a_F(T_1(p^2)),
    a_F(T_2(p^2)), \ldots, a_F(T_{n-1}(p^2))\right).
\end{equation*}

\begin{theorem}
  \label{thm:tjp2}
  Let $F$ and $G$ be Siegel eigenforms of degree $n$ on $\Gamma_0(N)$, of
  respective weights $(\lambda_1\geq\lambda_2\geq\ldots\geq\lambda_n)$ and
  $(\mu_1\geq\mu_2\geq\ldots\geq\mu_n)$, satisfying
  \begin{equation*}
    \sum_{j=1}^n \lambda_j\neq \sum_{j=1}^n \mu_j.
  \end{equation*}

  If $p$ is a prime number not dividing $N$, then
  \begin{equation*}
    \mathcal{E}_F(p)\neq\mathcal{E}_G(p).
  \end{equation*}
\end{theorem}
\begin{proof}
  Assume, on the contrary, that $\mathcal{E}_F(p)=\mathcal{E}_G(p)$.
  Theorem~\ref{thm:main_relation} gives
  \begin{equation*}
    c_nT_n(p^2) = T(p)^2 - T(p^2) - \sum_{s=1}^{n-1}c_sT_s(p^2),
  \end{equation*}
  so for our two forms $F$ and $G$ we get
  \begin{align*}
    c_na_F(T_n(p^2)) &= a_F(T(p))^2 - a_F(T(p^2)) -
    \sum_{s=1}^{n-1}c_sa_F(T_s(p^2))\\
    c_na_G(T_n(p^2)) &= a_G(T(p))^2 - a_G(T(p^2)) - 
    \sum_{s=1}^{n-1}c_sa_G(T_s(p^2)).
  \end{align*}
  Our hypothesis then tells us that the two right hand sides are equal, so
  $c_na_F(T_n(p^2))=c_na_G(T_n(p^2))$, and since $c_n\neq 0$, we get
  $a_F(T_n(p^2))=a_G(T_n(p^2))$.

  Finally, since $p$ does not divide the level $N$, Lemma~\ref{lem:tnp2}
  gives us
  \begin{align*}
    a_F(T_n(p^2)) &= p^{(\sum_{j=1}^n\lambda_j)-n(n+1)/2}\\
    a_G(T_n(p^2)) &= p^{(\sum_{j=1}^n\mu_j)-n(n+1)/2},
  \end{align*}
  from which we conclude that $\sum\lambda_j=\sum\mu_j$, a contradiction.
\end{proof}

A special case of interest is that of scalar-valued forms:
\begin{corollary}
  Under the same notation as in Theorem~\ref{thm:tjp2}, suppose that $F$ and
  $G$ are scalar-valued, of respective weights $k_1\neq k_2$.  Then
  $\mathcal{E}_F(p)\neq\mathcal{E}_G(p)$.
\end{corollary}


\section{Distinguishing degree $2$ eigenforms via the operators $T(p^r)$}
\label{sect:tpr}

The structure of the spaces of scalar-valued Siegel modular forms of degree
$2$ and level $1$ is well-known thanks to results of Igusa; these results,
together with the interplay between Jacobi modular forms and Siegel modular
forms, have allowed the explicit decomposition of these spaces into
eigenspaces for the Hecke operators.  This approach was introduced by
Skoruppa~\cite{Skoruppa}
and exploited and refined by a number of authors, most recently
Raum~\cite{Raum}.
Bases for these spaces are computed as sets of explicit Fourier expansions,
and the effect of the operators $T(p)$, $T_j(p^2)$ on these Fourier expansions
can then be computed using the explicit formulas
from~\cite{AndrianovZhuravlev} or~\cite{HafnerWalling}.
The
results of the last section fit naturally in this setting.

However, the Fourier-expansion based approach is not the only way of obtaining
Hecke eigenforms.  In fact, when working with vector-valued Siegel modular
forms of weight\footnote{In the context of vector-valued Siegel modular forms of
    degree $2$, the weight is often expressed as the actual representation
    $\Sym^j(\CC^2)\otimes\det^k$, rather than via the highest weight vector
    notation $(\lambda_1\geq\lambda_2)$.  We prefer to stick with the latter
    for the sake of consistency with the degree $n$ situation treated in the
    previous sections.  The reader who prefers to think in terms of
    $\Sym^j\otimes\det^k$ can use the dictionary
    $j=\lambda_1-\lambda_2$, $k=\lambda_2$.}
$\Sym^j\otimes\det^k$, this approach only works for small
values of $j$.  In a series of papers
(\cite{FaberGeer1},~\cite{FaberGeer2},~\cite{BergstromFaberGeer-level2},~\cite{BergstromFaberGeer-deg3}), Bergstr\"om, Faber and van der Geer
have developed a completely different way of computing eigenforms, based on
the cohomology of local systems on the moduli space of abelian varieties, and
counting curves of certain type over finite fields.  Their method naturally
works with vector-valued Siegel modular forms, and has been implemented
successfully to the study of forms of level $1$ and degree $2$ or $3$, and to
forms of level $\Gamma_0(2)$ and degree $2$.  The eigenvalues that it produces
are attached to the Hecke operators $T(p^r)$, rather than the $T_j(p^2)$ that
we have been considering so far.

In this section, we exhibit a result (Theorem~\ref{thm:tpr}) of the same
flavour as Theorem~\ref{thm:tjp2}, but for the operators $T(p^r)$.  As we will
see, the situation is more complicated here, so we prefer to restrict our
treatment to forms of degree $2$.  The same method should apply to other small
degrees, at the expense of increasingly tedious algebraic manipulations.


We start by gathering some useful relations between the eigenvalues $a_F(p^r)$
of an eigenform $F$.

\begin{lemma} 
  \label{lem:ap_deg2}
  Let $F$ be a Siegel eigenform of degree $2$ on $\Gamma_0(N)$, of
  weight
  $(\lambda_1\geq\lambda_2)$.  Let $p$ be a prime not dividing $N$.  Let
  $a_F(p^r)$ be the Hecke eigenvalues of $F$ under the operators $T(p^r)$. 
  Then the following identity holds: 
  \begin{equation}
    \label{eq:ap3}
    a_F(p^3)-2a_F(p)a_F(p^2)+a_F(p)^3-p^{\lambda_1+\lambda_2-4}(p+1)a_F(p)=0
  \end{equation}

  Further, if $a_F(p)=0$ then 
  \begin{align}
    \label{eq:ap_odd}
    &a_F(p^{2i+1})=0 \qquad\text{for all } i=0,1,2,\dots,\\
    \label{eq:ap4}
    &a_F(p^4)-a_F(p^2)^2-p^{\lambda_1+\lambda_2-4}a_F(p^2)+p^{2\lambda_1+2\lambda_2-6}=0,\\
    \label{eq:ap6}
    &a_F(p^6)-a_F(p^4)a_F(p^2)-p^{\lambda_1+\lambda_2-4}a_F(p^4)+p^{2\lambda_1+2\lambda_2-6}a_F(p^2)=0.
  \end{align}
\end{lemma}
\begin{proof}
By definition, $T(p^r)F=a_F(p^r)F$ for all $r$.  Moreover, if 
$p\nmid N$, then by Lemma~\ref{lem:tnp2} we have
\begin{equation}
  \label{eq:tp_ap}
    T_2(p^2)F = p^{\lambda_1+\lambda_2-6}F. 
\end{equation}

By~\cite[Theorem 2]{Shimura}, the Hecke operators at $p$ satisfy the
relations summarised by 
\begin{equation}
  \label{eq:gen_fn_deg2}
  z^4f(1/z)\sum_{i=0}^{\infty} T(p^i)z^i=1-p^2T_2(p^2)z^2
\end{equation}
where $f$ is the degree $4$ polynomial (\emph{aka} local Euler factor at $p$)
\begin{equation*}
  f(X)=X^4-T(p)X^3+\left[T(p)^2-T(p^2)-p^2T_2(p^2)\right]X^2
    -p^3T(p)T_2(p^2)X+p^6T_2(p^2)^2.
\end{equation*}
Upon expansion of~\eqref{eq:gen_fn_deg2} and inspection of the degree $3$
terms, we obtain 
\begin{equation*}
  T(p^3)-2T(p)T(p^2)+T(p)^3-p^2(p+1)T(p)T_2(p^2)=0,
\end{equation*}
which gives rise to~\eqref{eq:ap3}.

For any $i\geq 0$ we can equate the coefficient of $z^{2i+1}$ on
either side of~\eqref{eq:gen_fn_deg2} to find that 
\begin{multline*}
  T(p^{2i+1})-T(p)T(p^{2i})+\left[T(p)^2-T(p^2)-p^2T_2(p^2)\right]T(p^{2(i-1)+1})\\
  -p^3T(p)T_2(p^2)T(p^{2i-2})+p^6T_2(p^2)T(p^{2(i-2)+1})=0
\end{multline*} 
(we harmlessly define $T(p^i)=0$ for $i<0$).  The claimed
equality~\eqref{eq:ap_odd} now follows by induction on $i\geq 0$ under the
additional hypothesis $a_F(p)=0$.
 
The terms of degree $4$ and $6$ respectively in the
identity~\eqref{eq:gen_fn_deg2} are
\begin{align*}
  &T(p^4)-T(p)T(p^3)+T(p)^2T(p^2)-T(p^2)^2\\
  &\qquad\qquad\qquad\qquad\qquad -p^2T_2(p^2)-p^3T(p)^2T_2(p^2)+p^6T_2(p^2)^2=0,\\
  &T(p^6)-T(p)\left[T(p^5)-T(p)T(p^4)+p^3T_2(p^2)T(p^3)\right]\\
  &\qquad\qquad\qquad\qquad\qquad -T(p^2)T(p^4)-p^2T_2(p^2)^2T(p^4)+p^6T_2(p^2)^2T(p^2)=0.
\end{align*}
Equalities~\eqref{eq:ap4} and~\eqref{eq:ap6} follow
from these expressions, together
with the relation~\eqref{eq:tp_ap} and the assumption $a_F(p)=0$.
\end{proof}
 
\begin{theorem}
  \label{thm:tpr}
  Let $F$ and $G$ be Siegel eigenforms of degree $2$ on
  $\Gamma_0(N)$, of respective weights $(\lambda_1\geq\lambda_2)$ and
  $(\mu_1\geq\mu_2)$, with 
  \begin{equation*}
    \lambda_1+\lambda_2\neq\mu_1+\mu_2.
  \end{equation*}
  Let $p$ be a
  prime not dividing $N$.  Let $a_F(p^r)$, $a_G(p^r)$ denote the Hecke
  eigenvalues of $F$, $G$ for the operators $T(p^r)$.
  \begin{enumerate}[(i)]
    \item If $a_F(p)$ and $a_G(p)$ are not both zero, then
      \begin{equation*}
        \left(a_F(p), a_F(p^2), a_F(p^3)\right) \neq
        \left(a_G(p), a_G(p^2), a_G(p^3)\right).
      \end{equation*}
    \item If $a_F(p)=a_G(p)=0$, then
      \begin{equation*}
        \left(a_F(p^2), a_F(p^4), a_F(p^6)\right) \neq
        \left(a_G(p^2), a_G(p^4), a_G(p^6)\right).
      \end{equation*}
  \end{enumerate}
\end{theorem}
\begin{proof} 
  \
  \begin{enumerate}[(i)]
    \item Assume $a_F(p)\neq 0$.  Suppose, on the contrary, that
      $a_F(p^r)=a_G(p^r)$ for each of $r=1,2,3$.
 
      By~\eqref{eq:ap3} from Lemma~\ref{lem:ap_deg2},
      \begin{align*}
        a_F(p^3)-2a_F(p)a_F(p^2)+a_F(p)^3-p^{\lambda_1+\lambda_2-4}(p+1)a_F(p)&=0,\\
        a_G(p^3)-2a_G(p)a_G(p^2)+a_G(p)^3-p^{\mu_1+\mu_2-4}(p+1)a_G(p)&=0.
      \end{align*}
      This implies that 
      \begin{equation*}
        (p+1)a_F(p)(p^{\lambda_1+\lambda_2-4}-p^{\mu_1+\mu_2-4})=0
      \end{equation*}
      and under the assumption that $a_F(p)\neq 0$ we conclude that
      $\lambda_1+\lambda_2=\mu_1+\mu_2$, a contradiction.

    \item Assume that $a_F(p^r)=a_G(p^r)=:a(p^r)$ for each $r\in\{1,2,4,6\}$
      and that $a(p)=0$.  Equality~\eqref{eq:ap4} of Lemma~\ref{lem:ap_deg2}
      then implies that 
      \begin{equation*} 
        a(p)^2(p^{\mu_1+\mu_2}-p^{\lambda_1+\lambda_2}) =
        \frac{1}{p^2}(p^{2\mu_1+2\mu_2}-p^{2\lambda_1+2\lambda_2}).
      \end{equation*}
      Recognising the difference of perfect squares in the right hand
      expression, we may (under our standing assumption that 
      $\lambda_1+\lambda_2\neq\mu_1+\mu_2$)
      conclude that 
      \begin{equation*}
        a(p^2)=\frac{p^{\lambda_1+\lambda_2}+p^{\mu_1+\mu_2}}{p^2}.
      \end{equation*}
      Under the same assumptions, an identical analysis beginning with
      equality~\eqref{eq:ap6} of Lemma~\ref{lem:ap_deg2} leads to the identity
      \begin{equation*}
        a(p^4)=\frac{a(p^2)(p^{\lambda_1+\lambda_2}+p^{\mu_1+\mu_2})}{p^2}=a(p^2)^2.
      \end{equation*}

      Applying~\eqref{eq:ap4} to $F$ and $G$ respectively we find that 
      \begin{align*}
        a(p^2)&=a_F(p^2)=
        \frac{p^{2\lambda_1+2\lambda_2-6}}{p^{\lambda_1+\lambda_2-4}}=
        p^{\lambda_1+\lambda_2-2}\\
        a(p^2)&=a_G(p^2)=
        \frac{p^{2\mu_1+2\mu_2-6}}{p^{\mu_1+\mu_2-4}}=p^{\mu_1+\mu_2-2}.
      \end{align*}
      Once again we have contradicted the assumption 
      $\lambda_1+\lambda_2\neq\mu_1+\mu_2$.
  \end{enumerate}
\end{proof}

\begin{corollary}
  Let $F$ and $G$ be as in Theorem~\ref{thm:tpr}.  There exists
  \begin{equation*}
    m\leq (2\log(N)+2)^6\qquad\text{such that }a_F(m)\neq a_G(m).
  \end{equation*}
\end{corollary}
\begin{proof}
  Use the fact that given $N$, there exists a prime $p\leq 2\log(N)+2$ that
  does not divide $N$.  (For a proof, see~\cite[Theorem~1]{Ghitza}.)
\end{proof}

\section{Distinguishing degree $2$ eigenforms via Fourier coefficients}
\label{sect:fourier}


  



Let $\rho\colon\GL_2\to\GL(V)$ be a polynomial representation.
Any $F\in M_\rho(\Gamma_0(N))$ has a multivariate Fourier expansion of the
form
\begin{equation*}
  F(z)=\sum_{S\in\cF(2)} c_F(S) q^S\qquad\text{with }c_F(S)\in V,
\end{equation*}
where
\begin{itemize}
  \item the variable $z$ is in $\HH^2$, i.e. a symmetric $2\times 2$ complex
    matrix with positive-definite imaginary part;
  \item the index set $\cF(2)$ consists of all matrices $S\in\GL(2,\QQ)$ that are
    symmetric, positive-semidefinite and \emph{half-integral}, that is, $S=(s_{ij})$ with
    $2s_{ij}\in\ZZ$ and $s_{ii}\in\ZZ$;
  \item we set
    \begin{equation*}
      q^S=e^{2\pi i\Tr(Sz)}.
    \end{equation*}
\end{itemize}

Arakawa~\cite{Arakawa} obtained results on Euler products for vector-valued
Siegel modular forms of degree $2$ (extending Andrianov's investigation of the
scalar-valued case in~\cite{Andrianov-deg2}).  In particular, he gave
simple explicit formulas for the Hecke action on \emph{certain} Fourier
coefficients:
\begin{theorem}[Arakawa {\cite[Proposition 2.3]{Arakawa}}]
  Let $S=\begin{pmatrix}a & b/2\\ b/2 & c\end{pmatrix}\in\cF(2)$ be:
  \begin{itemize}
    \item primitive, that is, $\gcd(a, b, c)=1$;
    \item such that $d=b^2-4ac$ is the discriminant of the imaginary quadratic
      field $K=\QQ(\sqrt{d})$;
    \item such that $K$ has class number $1$.
  \end{itemize}
  Let $p$ be a rational prime that is inert in $K$, and let $m$ be a positive
  integer coprime to $p$.  Then, for $F\in M_\rho(\Gamma_0(N))$ and any $r$,
  we have
  \begin{equation*}
    c_{T(p^r)F}(mS) = c_F(p^rmS).
  \end{equation*}
\end{theorem}

\begin{corollary}
  \label{cor:tpr_coeff}
  Let $p\equiv 3\pmod{4}$, $F\in M_\rho(\Gamma_0(N))$ and $r\geq 1$, then
  \begin{equation*}
    c_{T(p^r)F}(I) = c_F(p^rI).
  \end{equation*}
\end{corollary}
\begin{proof}
  The identity matrix $I$ corresponds to the quadratic form $x^2+y^2$, which
  gives the imaginary quadratic field $K=\QQ(\sqrt{-1})$.
\end{proof}

\begin{theorem}
  \label{thm:fourier}
  Let $F$ and $G$ be Siegel eigenforms of degree $2$ on $\Gamma_0(N)$, of
  respective weights $(\lambda_1\geq\lambda_2)$ and $(\mu_1\geq\mu_2)$
  satisfying
  \begin{equation*}
    \lambda_1+\lambda_2\neq\mu_1+\mu_2.
  \end{equation*}
  Suppose that at least one of the Fourier coefficients $c_F(I)$ and $c_G(I)$
  is nonzero.  Let $p$ be a prime $\equiv 3\pmod{4}$ not dividing $N$.  Then
  there exists $r$ with $0\leq r\leq 6$ such that
  \begin{equation*}
    c_F(p^rI)\neq c_G(p^rI).
  \end{equation*}
\end{theorem}
\begin{proof}
  We proceed by contradiction: suppose
  \begin{equation*}
    c_F(p^rI)=c_G(p^rI)\qquad\text{for }0\leq r\leq 6.
  \end{equation*}
  (In particular, $c_F(I)=c_G(I)\neq 0$.)

  By Corollary~\ref{cor:tpr_coeff}
  \begin{equation*}
    a_F(p^r)=\frac{c_{T(p^r)F}(I)}{c_F(I)}
    =\frac{c_{T(p^r)G}(I)}{c_G(I)}=a_G(p^r)
    \qquad\text{for }0\leq r\leq 6.
  \end{equation*}
  This contradicts Theorem~\ref{thm:tpr}.
\end{proof}

\begin{remark}
  The assumption that at least one of $c_F(I)$ and $c_G(I)$ is nonzero is
  essential to the proof.  It is likely that the $I$-th coefficient of
  any Siegel eigenform is nonzero, but there are no general results in this
  direction.  It has been conjectured that the first Fourier-Jacobi
  coefficient of a Siegel eigenform $F$ is nonzero, and it is conceivable that
  the ($n=1$, $r=0$)-th coefficient of a Jacobi eigenform is also nonzero,
  which would imply our condition $c_F(I)\neq 0$.

  A discussion of this issue features in Arakawa's work on the $L$-functions
  attached to Siegel eigenforms, where he also gives a construction of some
  eigenforms $F$ such that $c_F(I)\neq 0$, see~\cite[Section 4]{Arakawa}.
\end{remark}

\begin{corollary}
  Let $F$ and $G$ be as in Theorem~\ref{thm:fourier}.  There exists a matrix
  $S\in\cF(2)$ such that
  \begin{equation*}
    \det(S)\leq (3\log(N)+4)^{12}\qquad\text{and}\qquad
    c_F(S)\neq c_G(S).
  \end{equation*}
\end{corollary}

This follows from the following estimate:
\begin{proposition}
  Let $N\geq 1$ be an integer.  Let $p$ be the smallest prime $\equiv
  3\pmod{4}$, not dividing $N$.  Then
  \begin{equation*}
    p\leq 3\log(N)+4.
  \end{equation*}
\end{proposition}
\begin{proof}
  The cases $1\leq N< 40$ are settled by a quick case-by-case computation.

  So we can assume $N\geq 40$.  We proceed by contradiction: suppose $N$ is
  divisible by all primes $\equiv 3\pmod{4}$ that are less than or equal to
  $3\log(N)+4$.  Then
  \begin{equation*}
    N\geq \prod_{\substack{p\leq 3\log(N)+4\\p\equiv 3~(\text{mod }4)}} p,
  \end{equation*}
  so that
  \begin{equation}
    \label{eq:theta3}
    \log(N)\geq \sum_{\substack{p\leq 3\log(N)+4\\p\equiv 3~(\text{mod }4)}}
    \log(p)=\theta_3(3\log(N)+4)\geq\theta_3(3\log(N)),
  \end{equation}
  where $\theta_3$ denotes the following modification of Chebyshev's function:
  \begin{equation*}
    \theta_3(x)=\sum_{\substack{p\leq x\\p\equiv 3~(\text{mod }4)}} \log(p).
  \end{equation*}
  If $N\geq 40$ then $\log(N)\geq 11/3$, so by Lemma~\ref{lem:theta3} the right hand
  side of Equation~\eqref{eq:theta3} is $>\log(N)$, which is a contradiction.
\end{proof}

\begin{lemma}
  \label{lem:theta3}
  The function $\theta_3$ satisfies
  \begin{equation*}
    \theta_3(3x)>x\qquad\text{for all }x\geq\frac{11}{3}.
  \end{equation*}
\end{lemma}
\begin{proof}
  Ramar\'e and Rumely give the following explicit estimate for $\theta_3$
  (see~\cite[Theorems 1 and 2]{RamareRumely}):
  \begin{equation*}
    |\theta_3(x)-x/2|\leq\begin{cases}
      0.001119x & \text{for }x\geq 10^{10}\\
      1.780719\sqrt{x} & \text{for }x< 10^{10}.
    \end{cases}
  \end{equation*}
  We can lower the bound $10^{10}$ at the expense of a weaker estimate:
  \begin{equation*}
    |\theta_3(x)-x/2|\leq\begin{cases}
      0.16188x & \text{for }x\geq 11\\
      1.780719\sqrt{x} & \text{for }x< 11.
    \end{cases}
  \end{equation*}
  So
  \begin{equation*}
    \theta_3(3x)\geq \frac{3x}{2}-0.16188\cdot 3x\cong 1.014x>x
    \qquad\text{for }3x\geq 11.
  \end{equation*}
\end{proof}

\bibliographystyle{plain}
\bibliography{rallis}

\begin{thebibliography}{10}

\bibitem{Andrianov-deg2}
A.~N. Andrianov.
\newblock Euler products that correspond to {S}iegel's modular forms of genus
  {$2$}.
\newblock {\em Uspehi Mat. Nauk}, 29(3 (177)):43--110, 1974.

\bibitem{Andrianov}
A.~N. Andrianov.
\newblock {\em Quadratic forms and {H}ecke operators}, volume 286 of {\em
  Grundlehren der Mathematischen Wissenschaften}.
\newblock Springer-Verlag, Berlin, 1987.

\bibitem{AndrianovZhuravlev}
A.~N. Andrianov and V.~G. Zhuravl{\"e}v.
\newblock {\em Modular forms and {H}ecke operators}, volume 145 of {\em
  Translations of Mathematical Monographs}.
\newblock American Mathematical Society, Providence, RI, 1995.
\newblock Translated from the 1990 Russian original by Neal Koblitz.

\bibitem{Arakawa}
T.~Arakawa.
\newblock Vector-valued {S}iegel's modular forms of degree two and the
  associated {A}ndrianov {$L$}-functions.
\newblock {\em Manuscripta Math.}, 44(1-3):155--185, 1983.

\bibitem{BergstromFaberGeer-level2}
J.~Bergstr{\"o}m, C.~Faber, and G.~van~der Geer.
\newblock Siegel modular forms of genus 2 and level 2: cohomological
  computations and conjectures.
\newblock {\em Int. Math. Res. Not. IMRN}, 2008.

\bibitem{BergstromFaberGeer-deg3}
J.~Bergstr{\"o}m, C.~Faber, and G.~van~der Geer.
\newblock Siegel modular forms of degree three and the cohomology of local
  systems.
\newblock {\em Selecta Math.}, to appear.

\bibitem{ChowGhitza}
S.~Chow and A.~Ghitza.
\newblock Distinguishing eigenforms modulo a prime ideal.
\newblock {\em arXiv:1304.1832}, 2013.

\bibitem{FaberGeer1}
C.~Faber and G.~van~der Geer.
\newblock Sur la cohomologie des syst\`emes locaux sur les espaces de modules
  des courbes de genre 2 et des surfaces ab\'eliennes. {I}.
\newblock {\em C. R. Math. Acad. Sci. Paris}, 338(5):381--384, 2004.

\bibitem{FaberGeer2}
C.~Faber and G.~van~der Geer.
\newblock Sur la cohomologie des syst\`emes locaux sur les espaces de modules
  des courbes de genre 2 et des surfaces ab\'eliennes. {II}.
\newblock {\em C. R. Math. Acad. Sci. Paris}, 338(6):467--470, 2004.

\bibitem{FaltingsChai}
G.~Faltings and C.-L. Chai.
\newblock {\em Degeneration of abelian varieties}, volume~22 of {\em Ergebnisse
  der Mathematik und ihrer Grenzgebiete (3)}.
\newblock Springer-Verlag, Berlin, 1990.
\newblock With an appendix by David Mumford.

\bibitem{Freitag}
E.~Freitag.
\newblock {\em Siegelsche {M}odulfunktionen}, volume 254 of {\em Grundlehren
  der Mathematischen Wissenschaften}.
\newblock Springer-Verlag, Berlin, 1983.

\bibitem{Ghitza}
A.~Ghitza.
\newblock Distinguishing {H}ecke eigenforms.
\newblock {\em Int. J. Number Theory}, 7(5):1247--1253, 2011.

\bibitem{Gross}
B.~H. Gross.
\newblock On the {S}atake isomorphism.
\newblock In {\em Galois representations in arithmetic algebraic geometry
  ({D}urham, 1996)}, volume 254 of {\em London Math. Soc. Lecture Note Ser.},
  pages 223--237. Cambridge Univ. Press, Cambridge, 1998.

\bibitem{HafnerWalling}
J.~L. Hafner and L.~H. Walling.
\newblock Explicit action of {H}ecke operators on {S}iegel modular forms.
\newblock {\em J. Number Theory}, 93(1):34--57, 2002.

\bibitem{KacCheung}
V.~Kac and P.~Cheung.
\newblock {\em Quantum calculus}.
\newblock Universitext. Springer-Verlag, New York, 2002.

\bibitem{Kohnen}
W.~Kohnen.
\newblock On {F}ourier coefficients of modular forms of different weights.
\newblock {\em Acta Arith.}, 113(1):57--67, 2004.

\bibitem{Murty}
M.~Ram Murty.
\newblock Congruences between modular forms.
\newblock In {\em Analytic number theory ({K}yoto, 1996)}, volume 247 of {\em
  London Math. Soc. Lecture Note Ser.}, pages 309--320. Cambridge Univ. Press,
  Cambridge, 1997.

\bibitem{RamareRumely}
O.~Ramar{\'e} and R.~Rumely.
\newblock Primes in arithmetic progressions.
\newblock {\em Math. Comp.}, 65(213):397--425, 1996.

\bibitem{Raum}
M.~Raum.
\newblock Efficiently generated spaces of classical {S}iegel modular forms and
  the {B}\"ocherer conjecture.
\newblock {\em J. Aust. Math. Soc.}, 89(3):393--405, 2010.

\bibitem{Schulze-Pillot}
R.~Schulze-Pillot.
\newblock Siegel modular forms having the same {$L$}-functions.
\newblock {\em J. Math. Sci. Univ. Tokyo}, 6(1):217--227, 1999.

\bibitem{Shimura}
G.~Shimura.
\newblock On modular correspondences for {$\Sp(n,\ZZ)$} and their congruence
  relations.
\newblock {\em Proc. Nat. Acad. Sci. U.S.A.}, 49:824--828, 1963.

\bibitem{Skoruppa}
N.-P. Skoruppa.
\newblock Computations of {S}iegel modular forms of genus two.
\newblock {\em Math. Comp.}, 58(197):381--398, 1992.

\bibitem{Geer}
G.~van~der Geer.
\newblock Siegel modular forms and their applications.
\newblock In {\em The 1-2-3 of modular forms}, Universitext, pages 181--245.
  Springer, Berlin, 2008.

\end{thebibliography}
\end{document}